\documentclass[11pt,a4paper,draft]{amsart}

\usepackage{amsmath,amssymb,enumerate,amsthm,url}

\newcommand{\UP}{\blacktriangle}

\newcommand{\Down}{\triangledown}

\theoremstyle{plain}

\newtheorem{theorem}{Theorem}[section]
\newtheorem{proposition}[theorem]{Proposition}
\newtheorem{lemma}[theorem]{Lemma}

\theoremstyle{remark}

\begin{document}

\title[$\mathrm{IntGC}$ has FMP]{Intuitionistic logic with a Galois connection has the finite model property$^\star$}
\thanks{$^\star$Addendum to the article: Wojciech Dzik, Jouni J{\"a}rvinen, and Michiro Kondo, 
\emph{Intuitionistic propositional logic with {G}alois connections}, Logic Journal of the IGPL
\textbf{18} (2010), 837--858. }

\author{Wojciech Dzik}
\address{WD:~Institute of  Mathematics, University of Silesia, ul.~Bankowa~12, \mbox{40-007 Katowice}, Poland}
\email{dzikw@silesia.top.pl}

\author{Jouni J{\"a}rvinen}
\address{JJ:~\url{http://sites.google.com/site/jounikalervojarvinen/}}
\email{Jouni.Kalervo.Jarvinen@gmail.com}

\author{Michiro Kondo}
\address{MK:~School of Information Environment, Tokyo Denki University, Inzai, 270-1382, Japan}
\email{kondo@sie.dendai.ac.jp}

\begin{abstract}
We show that the intuitionistic propositional logic with a Galois connection (IntGC), introduced by the authors, 
has the finite model property.
\end{abstract}

\maketitle

\section{Introduction}\label{Sec_intro}

In \cite{DJK10}, we introduced the intuitionistic propositional logic with a Galois connection (IntGC). 
In addition to the  intuitionistic logic axioms and inference rule of modus ponens, IntGC
contains just two rules of inference mimicking the condition defining Galois connections. 
A \emph{Galois connection} between partially ordered sets $P$ and $Q$ consists of two maps 
$f \colon P \to Q$ and $g \colon Q \to P$ such that for all $a \in P$ and $b \in Q$, we have
$f(a) \le b$ if and only if $a \le g(b)$. 
Note that in the literature can be found two ways to define Galois connections -- the one adopted here, 
in which the maps are order-preserving, and the other, in which they are reversing the order. 

We proved in  \cite{DJK10} that IntGC is complete with respect to both Kripke style and algebraic semantics.
Our intention was also to show that IntGC has the \emph{finite model property} (FMP), that is,
for every formula which is not provable, there exists a finite counter Kripke model.
Together with the other results presented in the paper, this would imply that the following
assertions are equivalent for every IntGC-formula $A$:
\begin{enumerate}[\rm  (i)]
\item $A$ is provable.

\item $A$ is valid in any finite distributive lattice with an additive and normal operator $f$;

\item $A$ is valid in any finite distributive lattice with a multiplicative and co-normal operator $g$;

\item  $A$ is valid in any finite Kripke model for IntGC.
\end{enumerate}

Unfortunately, our proof of FMP presented in \cite{DJK10} is incomplete and has some faults.
For instance, we did not show that the  frame on which the filtration is defined really forms a 
required Kripke frame. Therefore, here we present a more complete proof based on improved filtration model.

The paper is organised as follows. In Section~\ref{Sec_IntGC}, we recall the syntax, Kripke semantics 
and Kripke completeness of IntGC. Section~\ref{Sec_FMP} is devoted to proving the finite model property of IntGC.

\section{Logic IntGC} \label{Sec_IntGC}

The language $\mathcal{L}$ of IntGC is constructed from a countable set of propositional
variables $P$ and the connectives $\neg$, $\to$, $\vee$, $\wedge$, $\UP$, $\Down$.
The constant \emph{true} is defined by $\top := p \to p$ for some fixed propositional variable 
$p \in P$, and the constant \emph{false} is defined by $\bot := \neg \top$.

The logic IntGC is the smallest logic in $\mathcal{L}$ that contains the intuitionistic 
propositional logic Int, and is closed under the rules of \emph{substitution},
\emph{modus ponens}, and the rules:
\begin{enumerate}[({GC}1)]
\item If $A \to \Down B$ is provable, then $\UP A \to B$ is provable.
\item If $\UP A \to B$ if provable, then $A \to \Down B$ is provable.
\end{enumerate}

It is known that the following rules are admissible in IntGC:
\begin{enumerate}[\rm ({r}1)]
\item If $A$ is provable, then $\Down A$ is provable. 
\item If $A \to B$ is provable, then $\Down A \to \Down B$  and $\UP A \to \UP B$ are provable.
\end{enumerate}
In addition, the following formulas are provable:
\begin{enumerate}[\rm ({f}1)]
\item $A \to \Down \UP A$ \ and \  $\UP \Down A \to A$.

\item $\UP A \leftrightarrow \UP \Down \UP A$ \ and \ $\Down A \leftrightarrow \Down \UP \Down A$.

\item $\Down \top$  \ and \  $\neg \UP \bot$.

\item $\Down (A \wedge B)  \leftrightarrow  \Down A \wedge \Down B$ \ and \
$\UP (A \vee B)  \leftrightarrow \UP A \vee \UP B$.

\item $\Down (A \to B) \to (\Down A \to \Down B)$.
\end{enumerate}

A structure $\mathcal{F} = (X, \le, R)$ is called a \emph{Kripke frame} of IntGC, if $X$ is a non-empty set, 
$\le$ is a preorder on $X$, and $R$ is a relation on $X$ such that
\begin{equation} \tag{$\star$} \label{Eq:frame} 
({\ge} \circ R \circ {\ge}) \subseteq R .
\end{equation}

Let $v$ be a function $v \colon P \to \wp(X)$ assigning to each propositional variable $p$ 
a subset $v(p)$ of $X$. Such functions are called \emph{valuations} and the pair 
$\mathcal{M} = (\mathcal{F}, v)$ is called an IntGC-\emph{model}.
For any $x \in X$ and $A \in \Phi$, we define a {\em satisfiability relation} in
$\mathcal{M}$ inductively by the following way:
\begin{align*}
x \models p &\iff x \in v(p), \\
x \models A \wedge B &\iff x \models A \mbox{ and } x \models A, \\
x \models A \vee B &\iff x \models A \mbox{ or } x \models A, \\
x \models A \to B &\iff \mbox{ for all } y \geq x, \  y \models A \mbox{ implies }  y \models B, \\
x \models \neg A &\iff \mbox{ for no }y \geq x \mbox{ does }  y  \models A, \\
x \models \UP A &\iff \mbox{ exists } y \mbox{ such that } x \, R \, y \mbox{ and }  y \models  A, \mbox{ and}\\
x \models \Down A &\iff \mbox{ for all } y, y \, R \, x \mbox{ implies }  y \models  A.
\end{align*}

Let $x \leq y$. If $x \models \UP A$, there exists  $z$ such that $x \, R \, z$ and $z \models A$.  
Now $y \geq x$, $x \, R \, z$, and $z \geq z$ imply $y \, R \, z$ by \eqref{Eq:frame}. 
Thus, $y \models \UP A$. Similarly, if $y \not \models \Down A$, then there exists $z$ such that 
$z \, R \, y$ and $z \not \models A$. Now $z \geq z$, $z \, R \, y$, and $y \geq x$ imply $z \, R \,x$. 
This means $x \not \models \Down A$. Hence, the frame is \emph{persistent}. 

An IntGC-formula $A$ is \emph{valid in a Kripke model} $\mathcal{M}$,
if $x \models A$ for all $x \in X$. The formula is \emph{valid in a Kripke frame} $\mathcal{F}$, 
if $A$ is valid in every model based on $\mathcal{F}$. The formula $A$ is \emph{Kripke valid}
if $A$ is valid in every frame.

We proved in \cite{DJK10} that every formula is Kripke valid if and only if it is provable.

\section{$\mathrm{IntGC}$ has FMP} \label{Sec_FMP}

Let $A$ be a formula that is not provable. Then,  there exists a Kripke model 
$\mathcal{M} = (X,\leq,R)$ such that $A$ is not valid in $\mathcal{M}$.
We construct a counter model for $A$ on a finite frame.

Let $\mathrm{Sub}(A)$ be the set of subformulas of $A$. We define the set 
\[
\Gamma = \mathrm{Sub}(A) \cup \{ \Down \UP B \mid \UP B \in \mathrm{Sub}(A) \}
                         \cup \{ \UP \Down B \mid \Down B \in \mathrm{Sub}(A) \}.
\]
From this set, we can now define the set
\begin{align*}
\Sigma = \mathrm{Sub}(A) & \cup \{ (\Down \UP)^n \Down B \mid n \geq 0 \mbox{ and } \Down B \in \Gamma \} \\
                         & \cup \{ \UP (\Down \UP)^n \Down B \mid n \geq 0 \mbox{ and } \Down B \in \Gamma \} \\
                         & \cup \{ (\UP \Down)^n \UP B \mid n \geq 0 \mbox{ and } \UP B \in \Gamma \} \\
                         & \cup \{ \Down (\UP \Down)^n \UP B \mid n \geq 0 \mbox{ and } \UP B \in \Gamma \}.
\end{align*}
Obviously, $\mathrm{Sub}(A) \subseteq \Gamma \subseteq \Sigma$.

\begin{lemma} \label{Lem:Increasing}
\begin{enumerate}[\rm (a)]
\item If $\Down B \in \Sigma$, then $\UP \Down B\in \Sigma$.  
\item If $\UP B\in \Sigma$, then $\Down \UP B \in \Sigma$.
\end{enumerate}\end{lemma}

\begin{proof}(a)
Suppose that $\Down B \in \Sigma$. If $\Down B$ is of the form $(\Down \UP)^n \Down C$ for some $n \geq 0$, where 
$\Down C \in \Gamma$, then $\UP \Down B = \UP (\Down \UP)^n \Down C$ belongs to $\Sigma$ by definition. 
If $\Down B$ has the form of $\Down (\UP \Down)^m \UP C$ for some $m \geq 0$ where $\UP C \in \Gamma$, then
by the definition, $\UP \Down B = \UP \Down (\UP \Down)^m \UP C = (\UP\Down)^{m+1}\UP C$ is in $\Sigma$.

Assertion (b) can be proved analogously.
\end{proof}

A set of IntGC-formulas $\Sigma$ is said to be \emph{closed under subformulas} if $B \in \Sigma$ and $C \in \mathrm{Sub}(B)$  
imply $C \in \Sigma$.

\begin{lemma} \label{Lem:Subformula}
The set $\Sigma$ is closed under subformulas.
\end{lemma}

\begin{proof} 
Let $B \in \Sigma$. If $B$ is of the form $C \vee D$, $C \wedge D$, $C \to D$, or $\neg C$, then
$B$ must be in $\mathrm{Sub}(A)$ by the definition of $\Sigma$. Thus,
$C,D \in \textrm{Sub}(A) \subseteq \Sigma$.

If $\UP B \in \Sigma$ is of the form $\UP (\Down \UP)^n \Down C$ for some  $\Down C \in \Gamma$ and $n \geq 0$,
then  $B =  (\Down \UP)^n \Down C  \in \Sigma$.

If $\UP B \in \Sigma$ has the form $(\UP \Down)^n \UP C$ for some $n \geq 0$ and  $\UP C \in \Gamma$, then
$\UP B = \UP (\Down \UP)^n \Down D$, since $\UP C \in \Gamma$ means that
$C = \Down D  \in \mathrm{Sub}(A)$. Then, $B =  (\Down \UP)^n \Down D \in \Sigma$.
 
The remaining two cases are proved analogously.
\end{proof}

We now define for every formula $B \in \Sigma$, a unique formula $B^* \in \Gamma$ as follows:
\begin{enumerate}[(i)]
\item If $B \in \mathrm{Sub}(A)$ and $B$ is not of the form $\Down C$ nor $\UP C$, then $B^* = B$.
\item If $B$ is of the form $(\Down \UP)^n \Down C$, where $\Down C \in \Gamma$, then $B^* = \Down C$.
\item If $B$ is of the form $\UP (\Down \UP)^n \Down C$, where $\Down C = \Down \UP D \in \Gamma$ for some 
$\UP D \in \mathrm{Sub}(A)$, then $B^* = \UP D$.
\item If $B$ is of the form $(\UP \Down)^n \UP C$, where $\UP C \in \Gamma$, then $B^* = \UP C$.
\item If $B$ is of the form $\Down (\UP \Down)^n \UP C$, where $\UP C = \UP \Down D \in \Gamma$ for some
$\Down D \in \mathrm{Sub}(A)$, then $B^* = \Down D$.
\end{enumerate}

Related to the above definitions, we can write the following lemma.

\begin{lemma} \label{Lem:Equivalent}
For every $B \in \Sigma$, there exists a unique $B^* \in \Gamma $  such that the formula 
$B \leftrightarrow B^*$  is provable in {\rm IntGC}.
\end{lemma}

\begin{proof} We consider cases (ii) and (iii) only.

(ii) $B = (\Down \UP)^n \Down C = (\Down \UP)^{n-1} \Down \UP \Down C \leftrightarrow 
(\Down \UP)^{n-1} \Down C \leftrightarrow \cdots \leftrightarrow \Down  C = B^*$,
because $\Down C \leftrightarrow \Down \UP \Down C$ by (f2).

(iii) $B =  \UP (\Down \UP)^n \Down C =  \UP \Down \UP (\Down \UP)^{n-1}  \Down C \leftrightarrow
 \UP (\Down \UP)^{n-1}  \Down C \leftrightarrow \cdots \leftrightarrow \UP \Down C =
\UP \Down \UP D = \UP D = B^*$, since $\UP A \leftrightarrow \UP \Down \UP A$ for any $A$.
\end{proof} 

Lemma~\ref{Lem:Equivalent} says that since $\Gamma$ is finite, also the set $\Sigma$ can be considered ``finitary'',
because it can be divided into classes of provably equivalent formulas such that each class corresponds 
to one formula of $\Gamma$.

Now we define an equivalence $\sim$ on the set $X$ by setting 
\[ x \sim y \iff (\forall B \in \Sigma)\, x \models B \mbox{ iff } y \models B. \]
This means that points $x$ and $y$ are equivalent if they satisfy exactly the
same formulas of $\Sigma$.
We denote by $[x]$ the ${\sim}$-class of $x$, and $X/{\sim}$ is the set of
all ${\sim}$-classes.

\begin{lemma} \label{Lem:Quotient}
The quotient set $X/{\sim}$ is finite.
\end{lemma}

\begin{proof} Let $x \in X$. For all $y \in X$, $[x] \ne [y]$ means that there exists
a formula $B \in \Sigma$ that ``separates'' $x$ and $y$, that is, either (i) $x \models B$
and $y \not \models B$, or (ii) $y \models B$ and $x \not \models B$.
For instance, in case (i) this means by Lemma~\ref{Lem:Equivalent} that  $
x \models B^*$, $y \not \models B^*$, and $B^* \in \Gamma$. Because the set $\Gamma$ 
is finite, only a finite number of classes can be ``separated'' from $[x]$. 
Hence, also the quotient set $X/{\sim}$ must be finite.
\end{proof}

We denote  $X/{\sim}$ simply by $X^f$. We define in $X^f$  the relations $\leq^f$ and $R^f$ by setting:
\begin{align*}
[x] \le^f [y] &\iff (\forall B \in \Sigma) \, x \models  B \mbox{ implies } y \models B; \\
[x] \, R^f \, [y] &\iff (\forall B \in \Sigma) \,  \Down B \in \Sigma \mbox{ and } y \models \Down B \mbox{ imply } x \models B.
\end{align*}

We can now write the following lemma.

\begin{lemma}\label{Lem:RelProperties}
\begin{enumerate}[\rm (a)]
\item If \ $x\le y$, then $[x] \le^f [y]$.
\item If \ $x \, R \, y$, then $[x] \, R^f \, [y]$.
\end{enumerate}
\end{lemma}

\begin{proof} Claim (a) is obvious, because our Kripke frames are persistent.

(b) Assume $x \, R \,y$, $B \in \Sigma$, and $\Down B \in \Sigma$. 
By Lemma~\ref{Lem:Increasing}, also $\UP \Down B \in \Sigma$.
If $y \models \Down B$, then $x \,R \,y$ gives $x \models \UP \Down B$. We have
$x \models B$, because $\UP \Down B \to B$ is a valid formula. Hence, $|x| \, R^f \, |y|$.
\end{proof}

\begin{lemma}\label{Lem:FitratrationFrame}
The structure $\mathcal{F}^f = (X^f,\le^f,R^f)$ is a Kripke frame.
\end{lemma}

\begin{proof} It is clear that $\le^f$ is a preorder. Therefore, it is enough to show that 
\[ ({\ge}^f \circ R^f \circ {\ge}^f) \subseteq R^f .\]
Suppose that $[x] \geq^f [y]$, $[y]  \, R^f [z]$, and $[z] \geq^f [w]$.
For all $B \in \Sigma$, if $\Down B \in \Sigma$ and $w \models \Down B$, 
then $z \models \Down B$ because $[z] \geq^f [w]$. Now
$y \models B$ by  $[y]  \, R^f [z]$. Finally, $[x] \ge^f [y]$ implies
$x \models B$.  Thus, $[x] \, R^f [w]$,
\end{proof}

Our next lemma gives another condition for $R^f$.

\begin{lemma} \label{Lem:Condition}
$[x] \, R^f [y] \iff (\forall B \in \Sigma) \,  \UP B \in \Sigma \mbox{ and } y \models B \mbox{ imply } x \models \UP B$.
\end{lemma}

\begin{proof}
Let $B \in \Sigma$.
Assume $[x] \, R^f [y]$, $\UP B \in  \Sigma$ and $y \models B$. 
Since $B \to \Down \UP B$ is a provable formula, we have $y \models B \to \Down \UP B$ 
and so $y \models \Down \UP B$. Because $\UP B \in \Sigma$, 
Lemma~\ref{Lem:Increasing} gives $\Down \UP \in \Sigma$.
Since $[x] \, R^f [y]$, we get $x \models \UP B$.

Conversely, assume  that the right-side of the condition holds. If $\Down B \in \Sigma$  and  
$y \models \Down B$, then by  Lemma~\ref{Lem:Increasing}, $\UP \Down B \in \Sigma$, from which
we get $x \models \UP \Down B$ by the assumption. Because $\UP \Down B \to B$ is a provable
formula, we have $x \models B$. Thus, $[x] \, R^f [y] $.
\end{proof}

We define the valuation $v^f$  in such a way that for all proposition variables $p \in \Sigma$:
\[ v^f(p) = \{ [x] \mid x \models p \} . \]
Then, $\mathcal{M}^f = (X^f, \leq^f, R^f, v^f)$ is called \emph{filtration of $\mathcal{M}$ through $\Sigma$}.

\begin{lemma} \label{Lem:Filtration}
For any $B \in \Sigma$ and $x \in X$, $x \models B \ \mbox{iff}  \ [x] \models B$.
\end{lemma}

\begin{proof} By induction on $B$. This can be done, because $\Sigma$ is closed under subformulas.
The base case follows immediately from the definition of $v^f$, and
with respect to $\vee$ and $\wedge$ the proof is obvious. 

(i) Let $B$ of the form $\neg C  \in \Gamma$. Assume $[x] \models \neg C$. If $x \not \models \neg C$, then
there exists $y \geq x$ such that $y \models C$. Since $\Gamma$ is closed under subformulas,
also $C \in \Gamma$ and $[y] \models C$ by the induction hypothesis. Because $y \geq x$,
we have $[y] \geq^f [x]$ by Lemma~\ref{Lem:RelProperties}. This gives that 
$[x] \not \models \neg C$, a contradiction. So, $x \models \neg C$.

Conversely, suppose that $x \models \neg C$. Because $C \in \Gamma$, then by the
definition, $[y] \geq^f [x]$ implies $y \models \neg C$ and $y \not \models C$. 
By the induction hypothesis, we have that $[y] \geq^f [x]$ implies $[y] \not \models C$, 
that is, $[x] \models \neg C$.

(ii) Let $B$ of the form $C \to D \in \Gamma$. Assume $x \models C \to D$ and $[x] \not \models C \to D$.
Then, there exists $[y] \ge^f [x]$ such that $[y] \models C$, but $[y] \not \models D$.
By induction hypothesis, $y \models C$ and $y \not \models D$. Therefore,
$y \not \models C \to D$, which is impossible because  $[y] \ge^f [x]$.
 Thus, $[x] \models C \to D$.

On the other hand, if $[x] \models C \to D$, then for all $[y] \ge^f [x]$, $[y] \models C$ implies
$[y] \models D$. If $x \not \models C \to D$, there exists $y \ge x$ such that $y \models C$
and $y \not \models D$. Now $y \geq x$ gives $[y] \ge^f [x]$, and $[y] \models C$ and
$[y] \not \models D$ by the induction hypothesis. But this is impossible. So,
$x \models C \to D$.

(iii) Let $B$ be of the form $\Down C \in \Gamma$. 
Assume that $x \models \Down C$. If $[y] \, R^f \, [x]$, then $y \models C$,
and $[y] \models C$ follows from the induction hypothesis. Hence, $[x] \models \Down C$.

Conversely, assume that $[x] \models \Down C$ and $y \, R \, x$. Then,  $[y] \, R^f \, [x]$ by
Lemma~\ref{Lem:RelProperties}, which gives $[y] \models C$. We obtain $y \models C$ by 
the induction hypothesis, and so $x \models \Down C$. 

(iv) Let $B$ be of the form $\UP C  \in \Gamma$. If $x \models \UP C$, then there exists $y$ such that $x \, R y$
and $y \models C$. By the induction hypothesis, $[y] \models C$. Since $x \, R y$, we have 
$[x] R^f [y]$ and $[x] \models \UP C$.

On the other hand, if  $[x] \models \UP C$, then there exists $y$ such that $[x] \, R^f \, [y]$ and
$[y] \models C$. This implies $y \models C$ by the induction hypothesis. By Lemma~\ref{Lem:Condition},
we get $x \models \UP C$.
\end{proof}

Finally, we may write the following proposition.

\begin{proposition} \label{Prop:FMP}
{\rm IntGC} has the finite model property and is decidable.
\end{proposition}

\begin{proof} Suppose that a formula $A$ is not provable. Then, there exists a model $\mathcal{M} = (X,\le,R)$
such that $A$ is not valid in $\mathcal{M}$. This means that there exists $x \in X$ such that $x \not \models A$.
We may define the set $\Sigma$ and the filtration of $\mathcal{M}$ through $\Sigma$ as above.
Because $A \in \Sigma$, then $[x] \not \models A$ by Lemma~\ref{Lem:Filtration}, and hence $A$ is not
valid in the finite model $\mathcal{M}^f$.

In addition, it is well known that if a logic is finitely axiomatised with the finite model property, then the logic is decidable. 
\end{proof}

\providecommand{\bysame}{\leavevmode\hbox to3em{\hrulefill}\thinspace}
\providecommand{\MR}{\relax\ifhmode\unskip\space\fi MR }
\providecommand{\MRhref}[2]{%
  \href{http://www.ams.org/mathscinet-getitem?mr=#1}{#2}
}
\providecommand{\href}[2]{#2}

\section*{Contact Addresses}

\end{document}